\documentclass[11pt]{scrartcl}

\usepackage[english]{babel}
\usepackage[utf8]{inputenc}		
\usepackage{lmodern}

\usepackage{url}

\newcommand{\isdef}{\ensuremath{\mathrel{\mathop:}=}}		

\newcommand{\llangle}{\ensuremath{\langle\!\langle}}
\newcommand{\rrangle}{\ensuremath{\rangle\!\rangle}}

\usepackage{enumitem}
\usepackage{amsfonts, amsmath, amssymb}
\usepackage{amsthm}			

\usepackage{pgf, tikz}
\usepackage{tikz-cd}
\usetikzlibrary{calc}
\usetikzlibrary{arrows, arrows.meta}
\usetikzlibrary{positioning}

\usepackage{wrapfig}
\usepackage{nicefrac}					
\usepackage{color}

\DeclareMathOperator{\Sym}{Sym}

\DeclareMathOperator{\dist}{dist}

\DeclareMathOperator{\Stab}{Stab}

\newcommand{\menge}[1]{\ensuremath{\mathbb{#1}}}
\newcommand{\N}{\menge{N}}
\newcommand{\Z}{\menge{Z}}

\renewcommand{\phi}{\varphi}
\renewcommand{\epsilon}{\varepsilon}
\renewcommand{\subset}{\subseteq}

\newtheorem{lemma}{Lemma}[section]		
\newtheorem{example}[lemma]{Example}		
\newtheorem{remark}[lemma]{Remark}
\newtheorem{theorem}[lemma]{Theorem}			
\newtheorem*{theorem*}{Theorem}			
\newtheorem{corollary}[lemma]{Corollary}		
\newtheorem{proposition}[lemma]{Proposition}		
\newtheorem{definition}[lemma]{Definition}	
  {\begin{proof}[Well--defined]}
  {\end{proof}}

\newcommand{\tikzAngleOfLine}{\tikz@AngleOfLine}
\def\tikz@AngleOfLine(#1)(#2)#3{%
  \pgfmathanglebetweenpoints{%
    \pgfpointanchor{#1}{center}}{%
    \pgfpointanchor{#2}{center}}
  \pgfmathsetmacro{#3}{\pgfmathresult}%
}

\newcommand{\triangleMarkedLine}[4][black]{
	\draw[#1] (#2) -- (#3);
	\tikzAngleOfLine(#2)(#3){\angle}
	\def\dir{#4}
	\def\dist{0.18}
	\ifnum\dir>0
		\draw[#1,-open triangle 60] ($(#2)!0.5!(#3)$) -- +(\angle-90:\dist);
	\else
		\draw[#1,-open triangle 60] ($(#2)!0.5!(#3)$) -- +(\angle+90:\dist);
	\fi
}

\usepackage{amsmath, amssymb, amsfonts} 

\usepackage{ifthen}

\usepackage{tikz}
\usetikzlibrary{calc}
\usetikzlibrary{positioning}
\usetikzlibrary{shapes}
\usetikzlibrary{patterns}

\usepackage{tikz-cd}

\makeatletter
\edef\texforht{TT\noexpand\fi
  \@ifpackageloaded{tex4ht}
    {\noexpand\iftrue}
    {\noexpand\iffalse}}
\makeatother

\makeatletter
\newif\iftikz@node@phantom
\tikzset{
  phantom/.is if=tikz@node@phantom,
  text/.code=%
    \edef\tikz@temp{#1}%
    \ifx\tikz@temp\tikz@nonetext
      \tikz@node@phantomtrue
    \else
      \tikz@node@phantomfalse
      \let\tikz@textcolor\tikz@temp
    \fi
}
\usepackage{etoolbox}
\patchcmd\tikz@fig@continue{\tikz@node@transformations}{%
  \iftikz@node@phantom
    \setbox\pgfnodeparttextbox\hbox{}
  \fi\tikz@node@transformations}{}{}
\makeatother

\tikzset{ 
    vertexNodePlain/.style = {fill=#1, shape=circle, inner sep=0pt, minimum size=2pt, text=none},
    vertexNodePlain/.default=gray,
    vertexPlain/labels/.style = {
        vertexNode/.style={vertexNodePlain=##1},
        vertexLabel/.style={gray}
    },
    vertexPlain/nolabels/.style = {
        vertexNode/.style={vertexNodePlain=##1},
        vertexLabel/.style={text=none}
    },
    vertexPlain/.style = vertexPlain/#1,
    vertexPlain/.default=labels
}
\tikzset{
    vertexNodeNormal/.style = {fill=#1, shape=circle, inner sep=0pt, minimum size=4pt, text=none},
    vertexNodeNormal/.default = blue,
    vertexNormal/labels/.style = {
        vertexNode/.style={vertexNodeNormal=##1},
        vertexLabel/.style={blue}
    },
    vertexNormal/nolabels/.style = {
        vertexNode/.style={vertexNodeNormal=##1},
        vertexLabel/.style={text=none}
    },
    vertexNormal/.style = vertexNormal/#1,
    vertexNormal/.default=labels
}
\tikzset{
    vertexNodeBallShading/pdf/.style = {ball color=#1},
    vertexNodeBallShading/svg/.style = {fill=#1},
    vertexNodeBallShading/.code = {
        \if\texforht
            \tikzset{vertexNodeBallShading/svg=#1!90!black}
        \else
            \tikzset{vertexNodeBallShading/pdf=#1}
        \fi
    },
    vertexNodeBall/.style = {shape=circle, vertexNodeBallShading=#1, inner sep=2pt, outer sep=0pt, minimum size=3pt, font=\tiny},
    vertexNodeBall/.default = orange,
    vertexBall/labels/.style = {
        vertexNode/.style={vertexNodeBall=##1, text=black},
        vertexLabel/.style={text=none}
    },
    vertexBall/nolabels/.style = {
        vertexNode/.style={vertexNodeBall=##1, text=none},
        vertexLabel/.style={text=none}
    },
    vertexBall/.style = vertexBall/#1,
    vertexBall/.default=labels
}
\tikzset{ 
    vertexStyle/.style={vertexNormal=#1},
    vertexStyle/.default = labels
}

\newcommand{\vertexLabelR}[4][]{
    \ifthenelse{ \equal{#1}{} }
        { \node[vertexNode] at (#2) {#4}; }
        { \node[vertexNode=#1] at (#2) {#4}; }
    \node[vertexLabel, #3] at (#2) {#4};
}
\newcommand{\vertexLabelA}[4][]{
    \ifthenelse{ \equal{#1}{} }
        { \node[vertexNode] at (#2) {#4}; }
        { \node[vertexNode=#1] at (#2) {#4}; }
    \node[vertexLabel] at (#3) {#4};
}

\newcommand{\edgeLabelColor}{blue!20!white}
\tikzset{
    edgeLineNone/.style = {draw=none},
    edgeLineNone/.default=black,
    edgeNone/labels/.style = {
        edge/.style = {edgeLineNone=##1},
        edgeLabel/.style = {fill=\edgeLabelColor,font=\small}
    },
    edgeNone/nolabels/.style = {
        edge/.style = {edgeLineNone=##1},
        edgeLabel/.style = {text=none}
    },
    edgeNone/.style = edgeNone/#1,
    edgeNone/.default = labels
}
\tikzset{
    edgeLinePlain/.style={line join=round, draw=#1},
    edgeLinePlain/.default=black,
    edgePlain/labels/.style = {
        edge/.style={edgeLinePlain=##1},
        edgeLabel/.style={fill=\edgeLabelColor,font=\small}
    },
    edgePlain/nolabels/.style = {
        edge/.style={edgeLinePlain=##1},
        edgeLabel/.style={text=none}
    },
    edgePlain/.style = edgePlain/#1,
    edgePlain/.default = labels
}
\tikzset{
    edgeLineDouble/.style = {very thin, double=#1, double distance=.8pt, line join=round},
    edgeLineDouble/.default=gray!90!white,
    edgeDouble/labels/.style = {
        edge/.style = {edgeLineDouble=##1},
        edgeLabel/.style = {fill=\edgeLabelColor,font=\small}
    },
    edgeDouble/nolabels/.style = {
        edge/.style = {edgeLineDouble=##1},
        edgeLabel/.style = {text=none}
    },
    edgeDouble/.style = edgeDouble/#1,
    edgeDouble/.default = labels
}
\tikzset{
    edgeStyle/.style = {edgePlain=#1},
    edgeStyle/.default = labels
}

\newcommand{\faceColorY}{yellow!60!white}   
\newcommand{\faceColorB}{blue!60!white}     
\newcommand{\faceColorC}{cyan!60}           
\newcommand{\faceColorR}{red!60!white}      
\newcommand{\faceColorG}{green!60!white}    
\newcommand{\faceColorO}{orange!50!yellow!70!white} 
\newcommand{\faceColorW}{white}             

\newcommand{\faceColor}{\faceColorY}
\newcommand{\faceColorSwap}{\faceColorC}

\tikzset{
    face/.style = {fill=#1},
    face/.default = \faceColor,
    faceY/.style = {face=\faceColorY},
    faceB/.style = {face=\faceColorB},
    faceC/.style = {face=\faceColorC},
    faceR/.style = {face=\faceColorR},
    faceG/.style = {face=\faceColorG},
    faceO/.style = {face=\faceColorO},
    faceW/.style = {face=\faceColorW}
}
\tikzset{
    faceStyle/labels/.style = {
        faceLabel/.style = {}
    },
    faceStyle/nolabels/.style = {
        faceLabel/.style = {text=none}
    },
    faceStyle/.style = faceStyle/#1,
    faceStyle/.default = labels
}
\tikzset{ face/.style={fill=#1} }
\tikzset{ faceSwap/.code=
    \ifdefined\swapColors
        \tikzset{face=\faceColorSwap}
    \else
        \tikzset{face=\faceColor}
    \fi
}

\usepackage{hyperref}

\newcommand{\colV}{red}
\newcommand{\colE}{blue}
\newcommand{\colF}{green!80!black}
\newcommand{\V}[1]{\textcolor{\colV}{#1}}
\newcommand{\E}[1]{\textcolor{\colE}{#1}}
\newcommand{\F}[1]{\textcolor{\colF}{#1}}

\newcommand{\rcosets}[2]{\cos(#1:#2)}
\newcommand{\geoAut}[1]{#1^{\#}}

\newcommand{\barycentricTetra}{
    \def\scale{0.9}
        \foreach \i in {0,1,2,3,4}{
                \foreach \j in {0,1,2,3,4}{
                    \coordinate (P\i\j) at ($(\scale*\i,0)+(60:\scale*\j)$);
                }
            }

            \coordinate (ZL) at (barycentric cs:P00=1,P20=1,P02=1);
            \coordinate (ZR) at (barycentric cs:P20=1,P40=1,P22=1);
            \coordinate (ZM) at (barycentric cs:P02=1,P22=1,P20=1);
            \coordinate (ZU) at (barycentric cs:P02=1,P22=1,P04=1);

            \foreach \z/\b in {ZU/P03, ZU/P13, ZU/ZM, ZM/ZL, ZM/ZR, ZL/P01, ZL/P10, ZR/P31, ZR/P30}{
                \draw[\colV] (\z) -- (\b);
            }
            \foreach \z/\b in {ZU/P04, ZU/P02, ZU/P22, ZM/P02, ZM/P22, ZM/P20, ZL/P02, ZL/P00, ZL/P20, ZR/P22, ZR/P20, ZR/P40}{
                \draw[\colE] (\z) -- (\b);
            }
            \foreach \z/\b in {P00/P04, P04/P40, P40/P00, P02/P22, P22/P20, P20/P02}{
                \draw[\colF] (\z) -- (\b);
            }
}

\author{Markus Baumeister\footnote{Lehr- und Forschungsgebiet Algebra, RWTH Aachen University, Pontdriesch 10-16, Aachen, Germany}\footnote{baumeister@mathb.rwth-aachen.de}}
\title{Characterisation of geodesic self--dual regular surface triangulations}

\begin{document}
\maketitle
Regular surface triangulations appear commonly
in the literature. The interpretation as
\textit{regular maps} 
(e.\,g. \cite{ConderDobcsanyi_DeterminingRegularMaps}) makes
it easy to apply group--theoretic arguments to classify
regular triangulations. The approach used by
Dress (\cite{DressTilings}) has a more combinatorial flavour
and focuses on the flags of the triangulation, together
with their adjacencies. By recording these adjacencies
as permutations, group--theoretic arguments similar to those 
for regular maps become applicable.
We denote triangulations of closed surfaces, where every vertex is incident to
exactly $d$ triangles, as \textbf{degree--$d$--surfaces}.

By applying the 
homomorphism from a triangle group to a degree--$d$--surface
(\cite{Conder_RegularMapsHypermapsEulerCharacteristic}), we obtain a
correspondence between triangulations and certain subgroups of
triangle groups (Subsection \ref{Subsect_TriangleGroups}).

In this paper, we focus on \textit{geodesic duality}, the external surface
symmetry called \textit{opp} in Wilsons classification 
\cite{Wilson_OperatorsOverRegularMaps}. 
Similar to \cite{ArchConder_TrinitySymmetry} and 
\cite{ArchRich_ConstructionClassificationSelfDualSphericalPolyhedra}, 
we want to characterise all triangulations that
are self--dual with respect to geodesic duality.
Using the correspondence
to triangle subgroups, we characterise the corresponding
subgroups instead (Subsection \ref{Subsect_GeodesicTriangleGroups}). 
We obtain that all of these subgroups
contain a particular normal subgroup.

Therefore, we can interpret a degree--$d$--surface as
homomorphic image of a quotient of a triangle group (Section 
\ref{Sect_GeodesicTriangleGroup}). To carry over
the characterisation of subgroups corresponding to triangulations, we
employ the voltage assignments from \cite{ArchConder_TrinitySymmetry}
(Section \ref{Sect_Uncollapsed}). This
culminates in our characterisation of all geodesic self--dual triangulations:

\begin{theorem*}
    Let $H_d \isdef \langle a,b,c \mid a^2,b^2,c^2,(ab)^3,(ac)^2,(bc)^d,(bac)^d\rangle$
    be the quotient of a triangle group
    with $d \geq 5$ and $d \neq 7$. 
    Let $\#: H_d \to H_d, \quad a \mapsto a,\quad b\mapsto b,\quad c\mapsto ac$.

    There is a 1--1--correspondence between geodesic self--dual degree--$d$--surfaces
    and conjugacy classes of subgroups 
    $V \leq H_d$ 
    such that
    \begin{enumerate}
        \item $g^{-1}Vg \cap X = \{1\}$ for all $g \in H_d$ and 
            $X \in \{ \langle a\rangle, \langle c\rangle, \langle ab\rangle, \langle ac\rangle, \langle bc\rangle \}$
        \item $\geoAut{V}$ is conjugate to $V$.
    \end{enumerate}
\end{theorem*}

In Section \ref{Sect_Classification}, we present
the full list of geodesic self--dual degree--$d$--surfaces
for $d < 10$.

\section{Concept of surfaces}\label{Section_DefSurface}
In this section, we give a combinatorial description of surface 
triangulations, that focusses on the \textbf{flags}
of the triangulation $T$ (i.\,e. triples consisting of incident vertex, edge, and face).
Geometrically, this corresponds to the barycentric subdivision of each face
in $T$.

The resulting ``small
triangles'' within the original faces can be identified with the \textbf{flags}
of $T$. For example,
the barycentric subdivision of the tetrahedron has the following form 
(we identify edges according to the arrows drawn), where the tetrahedron is
build from the equilateral triangles with green edges around the $f_i$.
    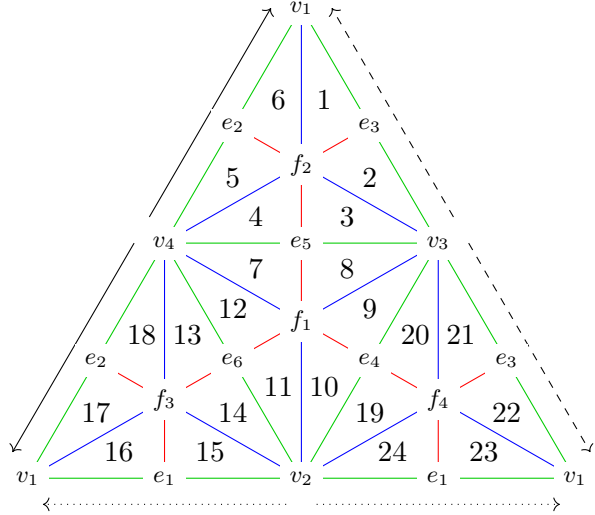
\begin{figure}[htbp]
        \begin{center}
            \begin{tikzpicture}[scale=2]
                \barycentricTetra
                
                \foreach \a/\b/\c/\n in 
                    {ZU/P04/P13/1, ZU/P22/P13/2, ZU/P22/P12/3, ZU/P02/P12/4, ZU/P03/P02/5, ZU/P03/P04/6,
                     ZM/P02/P12/7, ZM/P12/P22/8, ZM/P22/P21/9, ZM/P20/P21/10, ZM/P20/P11/11, ZM/P11/P02/12,
                     ZL/P02/P11/13, ZL/P11/P20/14, ZL/P20/P10/15, ZL/P10/P00/16, ZL/P00/P01/17, ZL/P01/P02/18,
                     ZR/P20/P21/19, ZR/P21/P22/20, ZR/P22/P31/21, ZR/P31/P40/22, ZR/P40/P30/23, ZR/P30/P20/24}{
                         \node at (barycentric cs:\a=1,\b=1,\c=1) {\n};
                     }

                \def\dist{0.2}
                \draw[->] ($(P02)+(120:\dist)$) -- ($(P04)+(180:\dist)$);
                \draw[->] ($(P02)+(180:\dist)$) -- ($(P00)+(120:\dist)$);
                
                \draw[->, dashed] ($(P22)+(60:\dist)$) -- ($(P04)+(0:\dist)$);
                \draw[->, dashed] ($(P22)+(0:\dist)$) -- ($(P40)+(60:\dist)$);

                \draw[->, dotted] ($(P20)+(-120:\dist)$) -- ($(P00)+(-60:\dist)$);
                \draw[->, dotted] ($(P20)+(-60:\dist)$) -- ($(P40)+(-120:\dist)$);

                \foreach \p in {P00,P10,P20,P30,P40,P01,ZL,P11,ZM,P21,ZR,P31,P02,P12,P22,P03,ZU,P13,P04}{
                    \fill[gray] (\p) circle (0.4pt);
                }

                \foreach \p/\n in {P00/$v_1$, P20/$v_2$, P40/$v_1$, P02/$v_4$, P22/$v_3$, P04/$v_1$}{
                    \node[fill=white] at (\p) {\footnotesize\n};
                }
                \foreach \p/\n in {P10/$e_1$, P30/$e_1$, P01/$e_2$, P11/$e_6$, P21/$e_4$, P31/$e_3$, P12/$e_5$, P03/$e_2$, P13/$e_3$}{
                    \node[fill=white] at (\p) {\footnotesize\n};
                }
                \foreach \p/\n in {ZL/$f_3$, ZM/$f_1$, ZR/$f_4$, ZU/$f_2$}{
                    \node[fill=white] at (\p) {\footnotesize\n};
                }

            \end{tikzpicture}
        \end{center}
        \caption{Barycentric subdivison of the tetrahedron}
        \label{Fig_Tetrahedron}
    \end{figure}

We can describe the adjacency structure of these flags with three 
involutions (previous works using this description include \cite{DressTilings} and
\cite{ArchRich_BranchedCoveringsOfMaps}).
The involution $\V{\alpha}$ maps the flag $(v,e,f)$ to the unique other flag
$(v',e,f)$. The involution $\E{\beta}$ does the same for the edges,
while $\F{\gamma}$ does it for the faces. For convenience, we label the
flags by natural numbers. In the example above we obtain:
    \begin{align*}
        \V{\alpha} &= (1,2)(3,4)(5,6)(7,8)(9,10)(11,12)(13,14)(15,16)(17,18)(19,20)(21,22)(23,24), \\
        \E{\beta} &= (1,6)(2,3)(4,5)(7,12)(8,9)(10,11)(13,18)(14,15)(16,17)(19,24)(20,21)(22,23), \\
        \F{\gamma} &= (1,22)(2,21)(3,8)(4,7)(5,18)(6,17)(9,20)(10,19)(11,14)(12,13)(15,24)(16,23).
    \end{align*}

The original vertices, edges, and faces of the triangulation $T$  
can be reconstructed from the involutions:
The vertices correspond to the orbits of $\langle\E{\beta},\F{\gamma}\rangle$,
the edges to the orbits of $\langle\V{\alpha},\F{\gamma}\rangle$, and the faces
to the orbits of $\langle\V{\alpha},\E{\beta}\rangle$ on the set of flags
(which we represent by their labels).

Since all faces of $T$ are triangles, $\V{\alpha}\E{\beta}$ consists only
of 3--cycles. Since we only consider closed surfaces (i.\,e. 
exactly two faces incident to each edge), $\V{\alpha}\F{\gamma}$ only 
consists of 2--cycles. In addition,
$\langle\V{\alpha},\E{\beta},\F{\gamma}\rangle$ should act transitively on
the set of flags (this corresponds to the strong 
connectivity\footnote{If for any two faces $f$, $g$ there is a sequence
of faces $f = f_1,f_2,\dots,f_n = g$ such that $f_i$ and $f_{i+1}$ are
incident to a common edge for all $1 \leq i < n$, the surface 
triangulation is called
\textit{strongly connected}.} of the surface triangulation).

For the purposes of this paper, we define \textit{surfaces} as follows:

\begin{definition}\label{Def_Surface}\label{Def_DegreeDSurface}
    Let $\mathcal{F}$ be a set (called \textbf{flags}) and 
    $\V{\alpha},\E{\beta},\F{\gamma}$ be 
    elements of the symmetric group on $\mathcal{F}$. The quadruple
    $(\mathcal{F},\V{\alpha},\E{\beta},\F{\gamma})$ is a \textbf{surface} if
    the following conditions hold:
    \begin{enumerate}
        \item $\V{\alpha},\E{\beta}$, and $\F{\gamma}$ are involutions that
            do not fix any flag in $\mathcal{F}$.
        \item $\langle\V{\alpha},\E{\beta},\F{\gamma}\rangle$ 
            acts\footnote{The group elements act from the left
            and we denote the action of $\alpha$ on $f$ by $\alpha.f$} transitively
            on $\mathcal{F}$.
        \item $\V{\alpha}\E{\beta}$ consists only of 3--cycles.
        \item $\V{\alpha}\F{\gamma}$ consists only of 2--cycles.
    \end{enumerate}
    We call the orbits of $\langle \E{\beta},\F{\gamma}\rangle$ on $\mathcal{F}$ 
    \textbf{vertices}, the orbits of $\langle\V{\alpha},\F{\gamma}\rangle$
    \textbf{edges}, and the orbits of
    $\langle \V{\alpha},\E{\beta}\rangle$ \textbf{faces}.
    A surface $(\mathcal{F},\V{\alpha},\E{\beta},\F{\gamma})$ is
    a \textbf{degree--$d$--surface} if $\E{\beta}\F{\gamma}$ only consists
    of $d$--cycles (for $d \geq 3$).
\end{definition}
We remark that this definition includes some cases that are not always
regarded as surface triangulations, for example two triangles sharing a boundary.

\begin{corollary}\label{Cor_OrbitLengthsOfDegreeDSurface}
    Let $(\mathcal{F},\V{\alpha},\E{\beta},\F{\gamma})$ be a 
    degree--$d$--surface. Then $\langle\V{\alpha},\E{\beta}\rangle$
    only has orbits of size 6 on $\mathcal{F}$, 
    $\langle\V{\alpha},\F{\gamma}\rangle$
    only has orbits of size 4, and $\langle\E{\beta},\F{\gamma}\rangle$
    only has orbits of size $2d$.
\end{corollary}

\section{Geodesic duality}\label{Sect_GeodesicDuality}
\def\skalPath{3.5}
Definition \ref{Def_Surface} allows the following duality:
If $(\mathcal{F},\V{\alpha},\E{\beta},\F{\gamma})$ is a surface, then
$(\mathcal{F},\V{\alpha},\E{\beta},\V{\alpha}\F{\gamma})$ is a surface as well
(in \cite[page 562]{Wilson_OperatorsOverRegularMaps}, this operation is called \textit{opp}).

\begin{definition}\label{Def_GeodesicDual}
    Let $S = (\mathcal{F},\V{\alpha},\E{\beta},\F{\gamma})$ be a surface. The
    surface $(\mathcal{F},\V{\alpha},\E{\beta},\V{\alpha}\F{\gamma})$ is 
    called its \textbf{geodesic dual} and denoted by $S^{\#}$.
\end{definition}

\begin{remark}
    Let $S$ be a surface, then $(S^\#)^\# = S$, justifying the name 
    \textit{duality}.
\end{remark}

\begin{example}
    The geodesic dual of the tetrahedron is a projective plane,
    illustrated in Figure \ref{Fig_ProjectivePlane}.
    \begin{figure}[htb]
        \begin{center}
            \begin{tikzpicture}[scale=2]
                \barycentricTetra
                
                \foreach \a/\b/\c/\n in 
                    {ZU/P04/P13/6, ZU/P22/P13/5, ZU/P22/P12/4, ZU/P02/P12/3, ZU/P03/P02/2, ZU/P03/P04/1,
                     ZM/P02/P12/7, ZM/P12/P22/8, ZM/P22/P21/9, ZM/P20/P21/10, ZM/P20/P11/11, ZM/P11/P02/12,
                     ZL/P02/P11/14, ZL/P11/P20/13, ZL/P20/P10/18, ZL/P10/P00/17, ZL/P00/P01/16, ZL/P01/P02/15,
                     ZR/P20/P21/20, ZR/P21/P22/19, ZR/P22/P31/24, ZR/P31/P40/23, ZR/P40/P30/22, ZR/P30/P20/21}{
                         \node at (barycentric cs:\a=1,\b=1,\c=1) {\n};
                     }

                \def\dist{0.2}
                \draw[<-, dashed] ($(P02)+(120:\dist)$) -- ($(P04)+(180:\dist)$);
                \draw[->, dotted] ($(P02)+(180:\dist)$) -- ($(P00)+(120:\dist)$);
                
                \draw[->] ($(P22)+(60:\dist)$) -- ($(P04)+(0:\dist)$);
                \draw[<-, dotted] ($(P22)+(0:\dist)$) -- ($(P40)+(60:\dist)$);

                \draw[<-] ($(P20)+(-120:\dist)$) -- ($(P00)+(-60:\dist)$);
                \draw[->, dashed] ($(P20)+(-60:\dist)$) -- ($(P40)+(-120:\dist)$);

                \foreach \p in {P00,P10,P20,P30,P40,P01,ZL,P11,ZM,P21,ZR,P31,P02,P12,P22,P03,ZU,P13,P04}{
                    \fill[gray] (\p) circle (0.4pt);
                }
                
                \foreach \p/\n in {P00/$w_1$, P20/$w_3$, P40/$w_2$, P02/$w_2$, P22/$w_3$, P04/$w_3$}{
                    \node[fill=white] at (\p) {\footnotesize\n};
                }
                \foreach \p/\n in {P10/$e_2$, P30/$e_3$, P01/$e_1$, P11/$e_6$, P21/$e_4$, P31/$e_1$, P12/$e_5$, P03/$e_3$, P13/$e_2$}{
                    \node[fill=white] at (\p) {\footnotesize\n};
                }
                \foreach \p/\n in {ZL/$f_3$, ZM/$f_1$, ZR/$f_4$, ZU/$f_2$}{
                    \node[fill=white] at (\p) {\footnotesize\n};
                }

            \end{tikzpicture}
        \end{center}
        \caption{Geodesic dual of the tetrahedron}
        \label{Fig_ProjectivePlane}
    \end{figure}
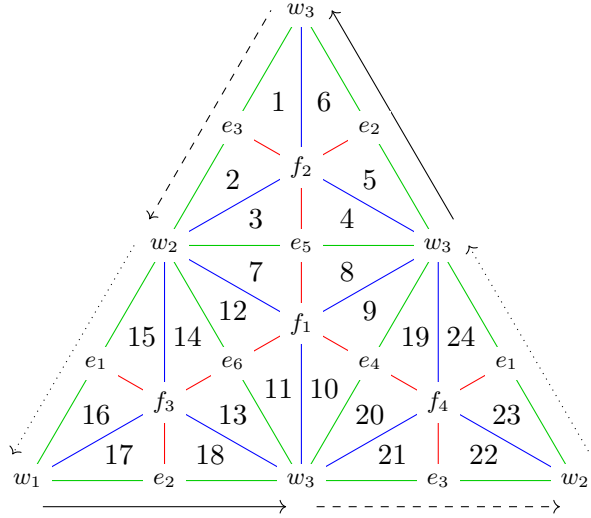
\end{example}

\section{Geometric interpretation of geodesic duality}
The definition of geodesic duality seems very ungeometric.
But, as the
name suggests, there is a deeper geometric meaning there. Fix
a flag $f \in \mathcal{F}$ and compare the actions of 
$\langle \E{\beta}\F{\gamma}\rangle$ and
$\langle \E{\beta}\V{\alpha}\F{\gamma}\rangle$ on $f$. Clearly,
$\langle \E{\beta}\F{\gamma}\rangle.f \subset \langle \E{\beta},\F{\gamma}\rangle.f$, 
so each orbit of 
$\langle \E{\beta}\F{\gamma}\rangle$ belongs to a unique vertex. If we consider the faces
belonging to the flags in $\langle \E{\beta}\F{\gamma}\rangle.f$, we obtain all faces
``around'' a vertex. We call such a set of faces an \textbf{umbrella} of the vertex
(illustrated in Figure \ref{Fig_Umbrella}).
    \begin{figure}[htb]
        \begin{center}
            \begin{tikzpicture}[scale=\skalPath]
                \coordinate (Z) at (0,0);
                \foreach \i in {0,...,5}{
                    \coordinate (P\i) at (-60+60*\i:1);
                    \coordinate (Q\i) at ($(Z)!0.5!(P\i)$);
                }
                \foreach \i in {0,...,4}{
                    \pgfmathparse{int(1+\i)}
                    \pgfmathsetmacro{\nxt}{\pgfmathresult}
     
                    \coordinate (R\i\nxt) at ($(P\i)!0.5!(P\nxt)$);
                    \coordinate (S\i\nxt) at (barycentric cs:Z=1,P\i=1,P\nxt=1);

                    \fill[yellow!40!white] (Z) -- (S\i\nxt) -- (Q\i) -- cycle;

                    \draw[\colV] (Q\i) -- (S\i\nxt) -- (Q\nxt) (S\i\nxt) -- (R\i\nxt);
                    \draw[\colE] (P\i) -- (S\i\nxt) -- (P\nxt) (S\i\nxt) -- (Z);
                    \draw[\colF] (Z) -- (P\i) -- (P\nxt) -- cycle;

                    \foreach \p in {Z,P\i,P\nxt,Q\i,Q\nxt,R\i\nxt,S\i\nxt}{
                        \fill[gray] (\p) circle (0.2pt);
                    }
                }

                \def\dist{0.1}
                \draw[->] ($(Z)+(-100:\dist)$) -- ($(P5)+(40:\dist)$);
                \draw[->] ($(Z)+(-80:\dist)$) -- ($(P0)+(140:\dist)$);

                \foreach \a/\b/\n in {Q4/S45/$f$, Q3/S34/$\E{\beta}\F{\gamma}.f$, Q2/S23/$(\E{\beta}\F{\gamma})^2.f$, Q1/S12/$(\E{\beta}\F{\gamma})^3.f$, Q0/S01/$(\E{\beta}\F{\gamma})^4.f$}{
                    \node at (barycentric cs:Z=1,\a=2,\b=2) {\n};
                }
            \end{tikzpicture}
        \end{center}
        \caption{An umbrella}
        \label{Fig_Umbrella}
    \end{figure}
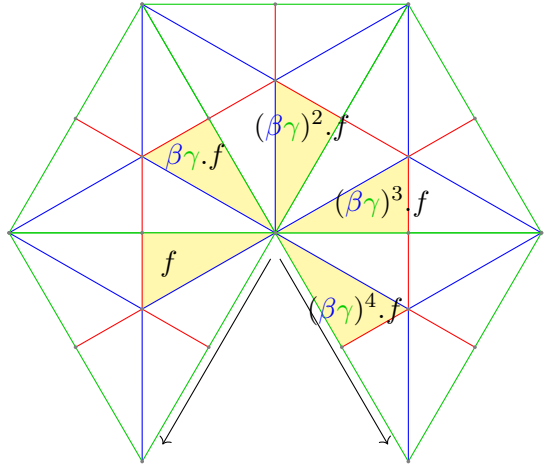

The geometric meaning of $\langle \E{\beta}\V{\alpha}\F{\gamma}\rangle.f$ is not that easily
apparent. Drawing the faces corresponding to the flags in
that orbit forms a ``straight'' strip of triangles (compare Figure \ref{Fig_Geodesic}).
On a purely
combinatorial level, these strips come closest to the notion of
``straight lines''. Therefore, we call these sets of faces \textbf{geodesics}.
    \begin{figure}[htb]
        \begin{center}
            \begin{tikzpicture}[scale=\skalPath]
                \coordinate (P4) at (0,0);
                \coordinate (P5) at (60:1);
                \coordinate (P6) at (1,0);
                \coordinate (P7) at ($(P5)+(P6)$);
                \coordinate (P2) at ($-1*(P6)$);
                \coordinate (P3) at ($(P2)+(P5)$);
                \coordinate (P1) at ($(P2)+(P3)$);

                \foreach \i in {2,...,6}{
                    \pgfmathparse{int(1+\i)}
                    \pgfmathsetmacro{\nxt}{\pgfmathresult}

                    \pgfmathparse{int(\i-1)}
                    \pgfmathsetmacro{\prv}{\pgfmathresult}

                    \coordinate (L\i) at ($(P\i)!0.5!(P\prv)$);
                    \coordinate (R\i) at ($(P\i)!0.5!(P\nxt)$);
                    \coordinate (O\i) at ($(P\prv)!0.5!(P\nxt)$);
                    \coordinate (M\i) at (barycentric cs:P\prv=1,P\i=1,P\nxt=1);

                    \draw[\colV] (L\i) -- (M\i) -- (R\i) (M\i) -- (O\i);
                    \draw[\colE] (P\prv) -- (M\i) -- (P\nxt) (M\i) -- (P\i);
                    \draw[\colF] (P\prv) -- (P\nxt) -- (P\i) -- cycle;

                    \fill[yellow!40!white] (M\i) -- (P\i) -- (R\i) -- cycle;

                    \foreach \p in {P\prv,P\nxt,P\i,L\i,R\i,M\i,O\i}{
                        \fill[gray] (\p) circle (0.2pt);
                    }
                }

                \def\dist{0.1}
                \draw[->] ($(P2)+(140:\dist)$) -- ($(P1)+(-80:\dist)$);
                \draw[->] ($(P7)+(-100:\dist)$) -- ($(P6)+(40:\dist)$);

                \foreach \i/\n in {2/$f$, 3/$\E{\beta}\V{\alpha}\F{\gamma}.f$, 4/$(\E{\beta}\V{\alpha}\F{\gamma})^2.f$, 5/$(\E{\beta}\V{\alpha}\F{\gamma})^3.f$, 6/$(\E{\beta}\V{\alpha}\F{\gamma})^4.f$}{
                    \node at (barycentric cs:P\i=1,M\i=2,R\i=2) {\n};
                }
            \end{tikzpicture}
        \end{center}
        \caption{A geodesic}
        \label{Fig_Geodesic}
    \end{figure}
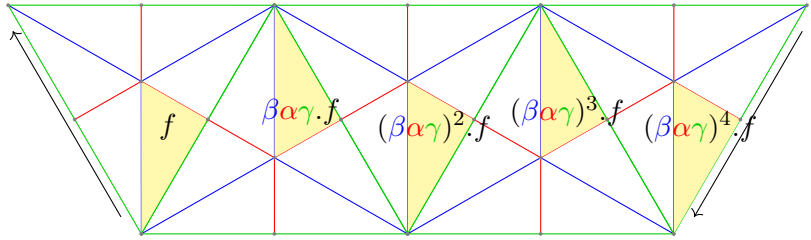

Since geodesic duality exchanges the orbits $\langle \E{\beta}\F{\gamma}\rangle.f$ and
$\langle \E{\beta}\V{\alpha}\F{\gamma}\rangle.f$, it also exchanges umbrellas and geodesics
in a surface. Heuristically, umbrellas are a local structure (to change
an umbrella, you have to change the vertex it corresponds to or one
of those adjacent to it), but geodesics show a global behaviour (if
any vertex is modified, the set of geodesics may change drastically).

Therefore, geodesic duality seems to exchange some local and global properties
in a given surface (and constructs a surface with inverted properties in the
process).
Since it relates very different
surfaces (like tetrahedron and projective plane),
one might hope to gain insight into one by analysing the other.

\section{Self--dual surfaces}\label{Sect_SelfDuality}
Given a notion of duality, a common approach is to analyse
self--dual objects. Here, we search for surfaces $S$ that are isomorphic
to their
geodesic dual $S^\#$.

\begin{definition}\label{Def_Isomorphism}
    Let $(\mathcal{F},\V{\alpha},\E{\beta},\F{\gamma})$ and 
    $(\mathcal{G},\V{\rho},\E{\sigma},\F{\tau})$ be two surfaces.

    A bijective map $\phi: \mathcal{F} \to \mathcal{G}$ is called
    an \textbf{isomorphism} if it satisfies
    \begin{align*}
        \phi^{-1}\V{\alpha}\phi = \V{\rho}, && \phi^{-1}\E{\beta}\phi = \E{\sigma}, && \phi^{-1}\F{\gamma}\phi = \F{\tau}.
    \end{align*}
    $(\mathcal{F},\V{\alpha},\E{\beta},\F{\gamma})$ is 
    \textbf{geodesic self--dual} if it is isomorphic to its geodesic
    dual $(\mathcal{F},\V{\alpha},\E{\beta},\V{\alpha}\F{\gamma})$.
\end{definition}

\subsection{Triangle groups and surface subgroups}\label{Subsect_TriangleGroups}
In this paper, we focus on degree--$d$--surfaces (compare Definition
\ref{Def_DegreeDSurface}), in order to exploit their nice group--theoretic
properties. They correspond to the triangulations in which
every vertex is incident to exactly $d$ faces and are 
intrinsically related to triangle groups 
(compare \cite[Section 6.2.8]{GrossTucker_TopologicalGraphTheory} for more 
details about triangle groups).
\begin{definition}\label{Def_TriangleGroup}
    For $d \in \N$, the \textbf{triangle group} is the finitely presented group
    \begin{equation}
        T_d \isdef \langle a,b,c \mid a^2, b^2, c^2, (ab)^3, (ac)^2, (bc)^d \rangle.
    \end{equation}
\end{definition}

\begin{remark}\label{Rem_TriangleGroupActsOnFlags}
    Let $(\mathcal{F},\V{\alpha},\E{\beta},\F{\gamma})$ be a 
    degree--$d$--surface. Then $T_d$ acts transitively on $\mathcal{F}$
    (from the left) via the permutation representation
    \begin{align*}
        T_d \to \Sym(\mathcal{F}) &&&& a \mapsto \V{\alpha} && b \mapsto \E{\beta} && c \mapsto \F{\gamma}.
    \end{align*}
\end{remark}

It is well--known (see for example
\cite[Theorem 6.3]{NeuStoyThomp94GroupsGeometry}, whose coset notation we
employ) that any transitive group action is 
equivariant to an action on cosets of a subgroup.
\begin{definition}\label{Def_Equivariance}
    Let $G$ be a group acting on $\Omega_1$ and
    $\Omega_2$. An \textbf{equivariance} between these
    actions is a bijection $\rho: \Omega_1 \to \Omega_2$
    fulfilling $\rho(g.\omega) = g.\rho(\omega)$ for
    all $g \in G$ and $\omega \in \Omega_1$.
\end{definition}
\begin{corollary}\label{Cor_SurfaceAsCosetAction}
    Let $(\mathcal{F},\V{\alpha},\E{\beta},\F{\gamma})$ be a 
    degree--$d$--surface, $f \in \mathcal{F}$, and
    $U \isdef \Stab_{T_d}(f)$. Then, 
    $\rcosets{T_d}{U} \to \mathcal{F},\; gU \mapsto g.f$ 
    is an equivariance between the action from Remark
    \ref{Rem_TriangleGroupActsOnFlags} and the left action
    of $T_d$ on $\rcosets{T_d}{U}$.
\end{corollary}

Every class of isomorphic degree--$d$--surfaces defines a conjugacy class
of subgroups. Our first goal is to classify all subgroups that
correspond to a degree--$d$--surface.

\begin{definition}\label{Def_SurfaceSubgroup}
    A subgroup $U \leq T_d$ is called \textbf{surface subgroup} if the
    coset action of $T_d$ on $\rcosets{T_d}{U}$ is equivariant
    to the action of $T_d$ on the flags of a degree--$d$--surface.
\end{definition}

\begin{lemma}\label{Lem_SurfaceSubgroupRegularityCharacterisation}
    The subgroup $U \leq T_d$ is a surface subgroup if and only if the groups
    $\langle a,b\rangle$, $\langle a,c\rangle$, and $\langle b,c\rangle$ act
    regularly on their coset orbits in $\rcosets{T_d}{U}$.

    In this case $(\rcosets{T_d}{U}, a, b, c)$ is a degree--$d$--surface.
\end{lemma}
\begin{proof}
    The subgroups $\langle a,b\rangle$, $\langle a,c\rangle$, and 
    $\langle b,c\rangle$ are dihedral groups of orders 6, 4, and $2d$,
    respectively. For example
    $\langle a,b\rangle \cong \langle x,y \mid x^2, y^2, (xy)^3\rangle \cong D_6$.

    Suppose first that $U \leq T_d$ is a surface subgroup.
    Then, there is a degree--$d$--surface 
    $(\mathcal{F},\V{\alpha},\E{\beta},\F{\gamma})$ 
    such that the actions of $T_d$
    on $\mathcal{F}$ and $\rcosets{T_d}{U}$ are equivariant.
    Since $\langle a,b\rangle$ acts on $\rcosets{T_d}{U}$ as
    $\langle \V{\alpha},\E{\beta}\rangle$ acts on $\mathcal{F}$
    (by Remark \ref{Rem_TriangleGroupActsOnFlags}), and
    $\langle \V{\alpha},\E{\beta}\rangle$ acts regularly on each 
    of its orbits, 
    $\langle a,b\rangle$ also acts regularly.
    Similar arguments apply to $\langle a,c\rangle$
    and $\langle b,c\rangle$.

    Conversely, we show that $(\rcosets{T_d}{U},a,b,c)$ is a
    degree--$d$--surface.
    \begin{itemize}
        \item Since $\langle a,b,c\rangle = T_d$, it acts transitively
            on $\rcosets{T_d}{U}$.
        \item By the definition of $T_d$, the elements $a$, $b$, and 
            $c$ are involutions.
            
            If $a$ fixed a coset $gU$, the orbits of $gU$ under 
            $\langle a,b\rangle$ would have size at most 2. Then,
            $\langle a,b\rangle$ could not be acting regularly on this
            orbit. Therefore, $a$ cannot fix a coset. The same
            argument applies to $b$ and $c$ as well.
        \item By the definition of $T_d$, we have $(ab)^3 = 1$.
            If there was a coset $gU$ such that $(ab)^k.gU = gU$ for
            some $0 < k < 3$, the group
            $\langle a,b\rangle$ would not be acting regularly on
            the orbit $\langle a,b\rangle.gU$. Therefore,
            $ab$ has only 3--cycles.

            The same argument applies to $ac$ and
            $bc$. \qedhere
    \end{itemize}
\end{proof}

We can characterise surface subgroups without reference to coset actions.
\begin{lemma}\label{Lem_SurfaceSubgroupIntersectionCharacterisation}
    A subgroup $U \leq T_d$ is a surface subgroup if and only if
    $gUg^{-1} \cap X = \{1\}$ for all $g \in T_d$ and all
    $X \in \{ \langle a, b\rangle, \langle a, c\rangle, \langle b, c\rangle \}$.
\end{lemma}
\begin{proof}
    We apply the characterisation from Lemma 
    \ref{Lem_SurfaceSubgroupRegularityCharacterisation}.

    The group $\langle a,b\rangle$ acts non--regularly if any 
    $1 \neq x \in \langle a,b\rangle$ fixes a coset $gU$, i.\,e.
    $x.gU = gU$. This is equivalent to $g^{-1}xg \in U$, or
    $x \in gUg^{-1}$. Therefore, the action is non--regular if and only if
    $gUg^{-1} \cap \langle a,b\rangle \neq \{1\}$. The arguments for
    $\langle a,c\rangle$ and $\langle b,c\rangle$ are similar.
\end{proof}

\subsection{Geodesic triangle groups}\label{Subsect_GeodesicTriangleGroups}
Our next goal is to characterise which surface subgroups correspond to
geodesic self--dual degree--$d$--surfaces.
\begin{lemma}\label{Lem_SurfaceSubgroupsThatAreGeodesicSelfDualImplyNormalSubgroup}
    Let $U \leq T_d$ be a surface subgroup such that
    $(\rcosets{T_d}{U},a,b,c)$
    is a geodesic self--dual degree--$d$--surface. 
    The normal closure
    $\llangle (bac)^d \rrangle$ is contained in $U$.
\end{lemma}
\begin{proof}
    $(bc)^d$ acts trivially on each coset $gU$ (for $g \in T_d$). By self--duality,
    $(bac)^d$ also acts trivially on $gU$.
    In other words, $(bac)^dgU = gU$, or $g^{-1}(bac)^dg \in U$ for all
    $g \in T_d$.
\end{proof}

Since the normal subgroup
$\llangle(bac)^d\rrangle$ is always contained in surface subgroups
of geodesic self--dual degree--$d$--surfaces, we can factor it out.
\begin{definition}\label{Def_GeodesicTriangleGroup}
    For every $d \in \N$, 
    the \textbf{geodesic triangle group} is defined as
    \begin{equation}
        H_d \isdef \langle a,b,c \mid a^2, b^2, c^2, (ab)^3, (ac)^2, (bc)^d, (bac)^d \rangle.
    \end{equation}
\end{definition}
These groups already appear in
\cite[Section 8.6]{CoxeterMoser_GeneratorsRelationsGroups}, under
the name $\{d,3\}_d$.

\begin{remark}\label{Rem_LatticeIsomorphism}
    There is a lattice isomorphism 
    $\{ \llangle (bac)^d \rrangle \leq U \leq T_d\} \to \{ V \leq H_d \}$.
    Since $\llangle (bac)^d\rrangle$ acts trivially on $\rcosets{T_d}{U}$,
    it is sufficient to consider the action of $T_d/\llangle (bac)^d\rrangle$
    on $\rcosets{T_d}{U}$. This action is equivariant to
    the action 
    of $H_d$ on $\rcosets{H_d}{V}$ 
    (with $V = U/\llangle (bac)^d\rrangle$).
    Therefore, $(\rcosets{T_d}{U},a,b,c)$ and $(\rcosets{H_d}{V},a,b,c)$ describe
    the same degree--$d$--surface.
\end{remark}
Geodesic duality can be formulated on the level of surface subgroups.

\begin{remark}\label{Rem_ProofGeodesicAutomorphism}
    We can define a group homomorphism $\#: H_d \to H_d$ by 
    \begin{align*}
        a \mapsto a && b \mapsto b && c \mapsto ac.
    \end{align*}
    To see that it is well--defined, 
    let $F$ be the free group generated by $\bar{a}$, $\bar{b}$, and $\bar{c}$.
    Then,
    \begin{equation*}
        \bar{\#}: F \to H_d, 
            \qquad \bar{a}\mapsto a,\quad \bar{b}\mapsto b,\quad \bar{c}\mapsto ac
    \end{equation*}
    is a well--defined group homomorphism. We
    consider its kernel. Since $\bar{\#}(\bar{a}) = a$ and 
    $\bar{\#}(\bar{b}) = b$, we immediately get
    $\llangle \bar{a}^2, \bar{b}^2, (\bar{a}\bar{b})^3 \rrangle \leq \ker{\bar{\#}}$.

    $\bar{c}^2$ and $(\bar{a}\bar{c})^2$ are mapped to $(ac)^2 = 1$ and $c^2 = 1$, so
    both lie in the kernel of $\bar{\#}$.
    An analogous argument
    shows that $(\bar{b}\bar{c})^d$ and $(\bar{b}\bar{a}\bar{c})^d$ lie in $\ker\bar{\#}$.
    Thus, $\bar{\#}$ factors over the normal subgroup generated by
    these relations (which gives $\#$).
\end{remark}
\begin{definition}\label{Def_GeodesicAutomorphism}
    The group homomorphism $\#$ from Remark \ref{Rem_ProofGeodesicAutomorphism} is
    called
    \textbf{geodesic automorphism}.
    For $g \in H_d$ and $V \leq H_d$, we employ the notation
    $\geoAut{g} \isdef \#(g)$ and $\geoAut{V} \isdef \{ \geoAut{g} \mid g \in V \}$.
\end{definition}

Now, we can characterise self--dual degree--$d$--surfaces group--theoretically.
\begin{proposition}\label{Prop_GeodesicDualOfDegreeDSurface}
    Let $V \leq H_d$ such that $S = (\rcosets{H_d}{V},a,b,c)$ is a 
    degree--$d$--surface. Then its geodesic dual is given by
    $(\rcosets{H_d}{\geoAut{V}},a,b,c)$.

    In particular, $S$ is geodesic self--dual if and only if $\geoAut{V}$ is 
    conjugate to $V$ in $H_d$.
\end{proposition}
\begin{proof}
    The geodesic dual of $S$ is $(\rcosets{H_d}{V},a,b,ac)$. This corresponds to the
    action of $H_d$ on the cosets of $V$ in $H_d$ via
    \begin{equation*}
        \phi: H_d \times \rcosets{H_d}{V} \to \rcosets{H_d}{V}, \qquad (h,tV) \mapsto \geoAut{h}tV.
    \end{equation*}
    We want to show that this action is equivariant to
    \begin{equation*}
        \psi: H_d \times \rcosets{H_d}{\geoAut{V}} \to \rcosets{H_d}{\geoAut{V}}, \qquad (h,t\geoAut{V}) \mapsto ht\geoAut{V}.
    \end{equation*}
    Since $\#$ is an automorphism of $H_d$, we have a bijection
    \begin{equation*}
        \#: \rcosets{H_d}{V} \to \rcosets{H_d}{\geoAut{V}}, \qquad tV \mapsto \geoAut{t}\geoAut{V}.
    \end{equation*}
    By Definition \ref{Def_Equivariance},
    we have to show that $\geoAut{\phi(h,tV)} = \psi( h,\geoAut{(tV)} )$:
    \begin{equation*}
        \geoAut{\phi(h,tV)} = \geoAut{ (\geoAut{h}tV) } = h\geoAut{(tV)} = \psi( h,\geoAut{(tV)}).
    \end{equation*}
    Therefore, the geodesic dual is given by $(\rcosets{H_d}{\geoAut{V}},a,b,c)$.
    This degree--$d$--surface is isomorphic to $(\rcosets{H_d}{V},a,b,c)$ if and
    only if $\geoAut{V}$ and $V$ are conjugate in $H_d$.
\end{proof}

\begin{corollary}\label{Cor_SelfDualSurfacesIntermediateCharacterisation}
    Let $U \leq T_d$. Then, $(\rcosets{T_d}{U},a,b,c)$ is a geodesic
    self--dual degree--$d$--surface if and only if
    \begin{itemize}
        \item $gUg^{-1} \cap X = \{1\}$ for all $g \in T_d$ and 
            $X \in \{\langle a,b\rangle, \langle a,c\rangle, \langle b,c\rangle\}$.
        \item $\llangle (bac)^d\rrangle \leq U$.
        \item $\geoAut{( U/\llangle (bac)^d\rrangle )}$ is conjugate
            to $U/\llangle (bac)^d\rrangle$ in $H_d$.
    \end{itemize}
    Furthermore, every geodesic self--dual degree--$d$--surface has this form.
\end{corollary}
\begin{proof}
    From Corollary \ref{Cor_SurfaceAsCosetAction} and Definition
    \ref{Def_SurfaceSubgroup} we deduce that every degree--$d$--surface 
    corresponds to a surface subgroup $U \leq T_d$ and can
    be represented as $S = (\rcosets{T_d}{U},a,b,c)$. Surface subgroups
    are characterised in Lemma 
    \ref{Lem_SurfaceSubgroupIntersectionCharacterisation}. This gives
    the first condition of the statement.

    If $S$ is geodesic self--dual,
    Lemma \ref{Lem_SurfaceSubgroupsThatAreGeodesicSelfDualImplyNormalSubgroup}
    gives the necessary condition $\llangle (bac)^d\rrangle \leq U$. This
    condition allows the reduction to $V \isdef U/\llangle(bac)^d\rrangle$ in 
    $H_d$ by Remark \ref{Rem_LatticeIsomorphism}.

    Then, Proposition \ref{Prop_GeodesicDualOfDegreeDSurface} gives the
    final condition of the statement.
\end{proof}
Since every element of the dihedral group 
$\langle x,y \mid x^2, y^2, (xy)^k\rangle$ is conjugate
to $x$, $y$ or $(xy)^m$ (with $m$ dividing $k$), 
we can replace the sets $X$ from Corollary 
\ref{Cor_SelfDualSurfacesIntermediateCharacterisation} by
\begin{equation}\label{Eq_SingleGeneratorGroups}
    X \in \{ \langle a\rangle, \langle b\rangle, \langle c\rangle, \langle ab\rangle, \langle ac\rangle, \langle bc\rangle \}.
\end{equation}

\section{Reduction to geodesic triangle groups}\label{Sect_GeodesicTriangleGroup}
Corollary \ref{Cor_SelfDualSurfacesIntermediateCharacterisation}
characterises geodesic self--dual degree--$d$--surfaces by considering
both groups,
$T_d$ and $H_d$. We would like to have a characterisation in which only
$H_d$ appears, since $H_d$ is often smaller than $T_d$.
\begin{remark}\label{Rem_SelfDualUniversalCoverGroupSizes}
    The groups $T_d$ and $H_d$ have the following orders:
    \begin{center}
        \begin{tabular}{l|cccccccccc}
            d       & 1 & 2     & 3 & 4 & 5     & 6 & 7 & 8                     & 9 & $\geq 10$ \\ \hline
            $|T_d|$ & 2 & 12    & 24 & 48 & 120 &$\infty$ & $\infty$ & $\infty$ & $\infty$ & $\infty$ \\ \hline
            $|H_d|$ & 1 & 4     & 1 & 4 & 60    & 108  & 1 & 672                & 3420 & $\infty$ \\
        \end{tabular}
    \end{center}
\end{remark}
\begin{proof}
    The finiteness results can be calculated very easily in GAP (\cite{GAP4}).
    They can also be found in
    \cite[Table 8]{CoxeterMoser_GeneratorsRelationsGroups}, if one uses the
    notation $H_d = \{d,3\}_d$.

    For $H_d$, Edjvet and Juhász show all of the results in \cite{EdjvetJuhaszGmnp}. In 
    comparison to our notation, the roles of $b$ and $c$ are interchanged. 
    Therefore, the parameters in their paper are set as follows: $m\isdef3$ and 
    $n=p\isdef d$.
\end{proof}

We conclude that there are no geodesic self--dual
degree--$d$--surfaces for $d \in \{3,4,7\}$.
\begin{corollary}\label{Cor_NoGeodesicSelfDualFor347}
    Let $U \leq T_d$ and 
    $\llangle (bac)^d\rrangle \cap \langle bc\rangle \neq \{1\}$.
    Then, $(\rcosets{T_d}{U},a,b,c)$ is not a geodesic
    self--dual degree--$d$--surface.

    In particular, there is no geodesic self--dual degree--$d$--surface
    for $d \in \{3,4,7\}$.
\end{corollary}
\begin{proof}
    If $(\rcosets{T_d}{U},a,b,c)$ was a geodesic self--dual
    degree--$d$--surface, Lemma 
    \ref{Lem_SurfaceSubgroupsThatAreGeodesicSelfDualImplyNormalSubgroup} 
    would give
    $\llangle (bac)^d\rrangle \leq U$. But
    $U \cap \langle bc\rangle \geq \llangle (bac)^d\rrangle \cap \langle bc\rangle \neq \{1\}$,
    which contradicts the characterisation of surface subgroups
    in Lemma \ref{Lem_SurfaceSubgroupIntersectionCharacterisation}.

    Since $H_3 = H_7 = \{1\}$, we have $\llangle (bac)^d\rrangle = T_d$ 
    in these cases. For $H_4$, it can be checked (either with GAP
    or with a calculation like in 
    \cite[Section 3.3]{BjornerBrenti_CombinatoricsOfCoxeterGroups}) that
    $c(bac)^4cb(bac)^4b=(cb)^2$.
\end{proof}

We would like to replace $gUg^{-1} \cap X = \{1\}$ for $g \in T_d$ by
$gVg^{-1} \cap X = \{1\}$ for $g \in H_d$.
\begin{lemma}\label{Lem_GroupIntersectionReduction}
    Let $G$ be a group, $W, X \leq G$ and $N \trianglelefteq G$ with
    $N \leq W$ and $X \cap N = \{1\}$. Then 
    $W \cap X \cong W/N \cap XN/N$.
\end{lemma}
\begin{proof}
    The result follows from the homomorphism theorems:
    \begin{align*}
        W \cap X \cong (W \cap X)/(W \cap X \cap N) &\cong (W \cap X)N/N \\
        &= WN/N \cap XN/N = W/N \cap XN/N. \qedhere
    \end{align*}
\end{proof}
To apply this lemma to
$G = T_d$, $W = gUg^{-1}$, $X = X$, and $N = \llangle (bac)^d\rrangle$,
we need to show that $\llangle (bac)^d\rrangle \cap X = \{1\}$ for
all $X$ from Equation \eqref{Eq_SingleGeneratorGroups}. This is 
easy for $X \neq \langle bc\rangle$.

\begin{lemma}\label{Lem_TrivialIntersectionWithNormalSubgroup}
    $\llangle (bac)^d\rrangle \cap X = \{1\}$ for 
    $X \in \{ \langle a \rangle, \langle b\rangle, \langle c\rangle, \langle ab\rangle, \langle ac\rangle \}$
    and $d \geq 5$ with $d \neq 7$.
\end{lemma}
\begin{proof}
    If $a \in \llangle (bac)^d\rrangle$, we also have
    $\llangle a\rrangle \leq \llangle (bac)^d\rrangle$. In particular,
    \begin{equation*}
        H_d \cong T_d / \llangle (bac)^d\rrangle \cong (T_d / \llangle a\rrangle) / ( \llangle (bac)^d\rrangle / \llangle a\rrangle ).
    \end{equation*}
    Since $T_d / \llangle a\rrangle = \langle b,c \mid b^2,c^2,(bc)^d\rangle$ 
    is a dihedral group of order $2d$, we 
    conclude $|H_d| < 2d$. By Remark 
    \ref{Rem_SelfDualUniversalCoverGroupSizes}, this cannot happen for
    $d \geq 5$ and $d \neq 7$.
    Similar arguments apply to $b$ and $c$.
    We can apply the same argument to $ab$ and $ac$.
    We get $T_d / \llangle ab\rrangle \cong D_4$ and 
    $T_d / \llangle ac\rrangle \cong D_6$, so
    $\llangle (bac)^d\rrangle \cap \langle ab\rangle$ and
    $\llangle (bac)^d\rrangle \cap \langle ac\rangle$ 
    are trivial as 
    well.
\end{proof}

\begin{corollary}\label{Cor_GeodesicSurfaceSubgroupCharacterisation}
    Let $V \leq H_d$ with $d \geq 5$ and $d \neq 7$.
    Then,
    $(\rcosets{H_d}{V},a,b,c)$ is a geodesic self--dual degree--$d$--surface
    if and only if
    \begin{itemize}
        \item $gVg^{-1} \cap X = \{1\}$ for 
            $X \in \{\langle a\rangle,\langle c\rangle,\langle ab\rangle,\langle ac\rangle,\langle bc\rangle\}$
            and all $g \in H_d$.
        \item $\geoAut{V}$ is conjugate to $V$.
        \item $\llangle (bac)^d\rrangle \cap \langle bc\rangle = \{1\}$ in $T_d$.
    \end{itemize}
    Furthermore, every geodesic self--dual degree--$d$--surface has this form.
\end{corollary}
\begin{proof}
    Let $(\rcosets{H_d}{V},a,b,c)$ be geodesic self--dual degree--$d$--surface.
    Then, there is a surface subgroup $U \leq T_d$ with
    $U/\llangle(bac)^d\rrangle = V$. By Corollary
    \ref{Cor_SelfDualSurfacesIntermediateCharacterisation},
    $gUg^{-1}\cap X = \{1\}$ for all $X$ in the list 
    \eqref{Eq_SingleGeneratorGroups}. By assumption and Lemma
    \ref{Lem_TrivialIntersectionWithNormalSubgroup}, we can
    apply Lemma \ref{Lem_GroupIntersectionReduction} to conclude
    $gVg^{-1} \cap X = \{1\}$.

    Conversely, there is an $U \leq T_d$ with $U/\llangle(bac)^d\rrangle=V$.
    Since $(ab)a(ab)^{-1}=b$, we also have $gVg^{-1}\cap\langle b\rangle =\{1\}$.
    Since 
    $V \cap \langle ac\rangle \cong \geoAut{(V \cap \langle ac\rangle)} = \geoAut{V}\cap \langle c\rangle = hVh^{-1}\cap\langle c\rangle$,
    for some $h \in H_d$, this intersection is also trivial. Applying
    Lemma \ref{Lem_GroupIntersectionReduction} shows that the conditions
    of Corollary \ref{Cor_SelfDualSurfacesIntermediateCharacterisation}
    are fulfilled.
\end{proof}

\section{Uncollapsed geodesic triangle groups}\label{Sect_Uncollapsed}
The characterisation of geodesic self--dual degree--$d$--surfaces in
Corollary \ref{Cor_GeodesicSurfaceSubgroupCharacterisation} contains the
assumption
$\llangle (bac)^d\rrangle \cap \langle bc\rangle = \{1\}$. In this section,
we show that this condition is not necessary. We start by rewriting it.

\begin{remark}\label{Rem_RelatorsInGeodesicTriangleGroup}
    Since $\#$ is an automorphism of $H_d$,
    $(bc)^k = 1$ if and only if $(bac)^k = 1$.
\end{remark}
\begin{lemma}\label{Lem_UncollapsedConditionForSelfDuality}
    In the triangle group $T_d$, we have
    $\llangle (bac)^d\rrangle \cap \langle bc\rangle= \langle (bc)^k\rangle$
    for $1 \leq k \leq d$
    if and only if $H_d = H_k$.
\end{lemma}
\begin{proof}
    Suppose first that 
    $\llangle (bac)^d\rrangle \cap \langle bc\rangle= \langle (bc)^k\rangle$.
    Clearly, $k$ divides $d$. Then:
    \begin{align*}
        H_d &= T_d / \llangle (bac)^d\rrangle \\
        &= T_d / \llangle (bc)^k, (bac)^d\rrangle \\
        &= \langle a,b,c \mid a^2, b^2, c^2, (ab)^3, (ac)^2, (bc)^d, (bc)^k, (bac)^d \rangle \\
        &= \langle a,b,c \mid a^2, b^2, c^2, (ab)^3, (ac)^2, (bc)^k, (bac)^d\rangle.
    \end{align*}
    Since $(bc)^k = 1$, Remark \ref{Rem_RelatorsInGeodesicTriangleGroup}
    implies $(bac)^k=1$ as well.

    For the other direction, we note that $T_d \neq T_k$ since
    the factor groups with respect to $\llangle a\rrangle$ are dihedral groups
    of different orders.
    Therefore, $H_d = H_k$ implies that $(bc)^k \in \llangle (bac)^d\rrangle$
    and $(bac)^k \in \llangle (bac)^d\rrangle$.
\end{proof}

This motivates the following definition.
\begin{definition}\label{Def_Uncollapsed}
    $H_d$ is \textbf{uncollapsed} if $H_d \neq H_k$ for all $1 \leq k \leq d$.
\end{definition}
We want to show that $H_d$ is uncollapsed if $d \geq 5$ and $d \neq 7$.

\begin{lemma}\label{Lem_UncollapsedReducedToPrimes}
    $H_d$ is uncollapsed if and only if $H_{\frac{d}{p}} \neq H_d$ for all
    primes $p$ dividing $d$.
\end{lemma}
\begin{proof}
    If $H_k = H_d$ with $\frac{d}{k}$ not prime, there is a prime $p$ dividing
    this fraction, such that
    \begin{equation*}
        H_{\frac{d}{p}} = H_d/\llangle (bc)^{\frac{d}{p}}, (bac)^{\frac{d}{p}}\rrangle =
            H_k/\llangle (bc)^{\frac{d}{p}}, (bac)^{\frac{d}{p}}\rrangle = H_k = H_d.\qedhere
    \end{equation*}
\end{proof}

\begin{corollary}\label{Cor_UncollapsedByFiniteFactor}
    Let $p \not\in \{3,7\}$ be a prime. Then $H_p$ is uncollapsed.
    Furthermore, $H_d$ is uncollapsed for $d \in \{2,6,12,15,21,35,49\}$.
\end{corollary}
\begin{proof}
    Apply Lemma \ref{Lem_UncollapsedReducedToPrimes}. We have
    $\{1\} = H_1 = H_p$ if and only if $p \in \{3,7\}$.

    If $H_d$ is finite, we can inspect the table from Remark
    \ref{Rem_SelfDualUniversalCoverGroupSizes} to see whether it
    is uncollapsed. If $H_d$ is infinite, but for every prime $p$ dividing
    $d$, the group $H_{\frac{d}{p}}$ is finite, then $H_d$ has to be
    uncollapsed.
\end{proof}

\begin{lemma}\label{Lem_DoublePrimeIsUncollapsed}
    Let $p$ be an odd prime. Then $H_{2p}$ is uncollapsed.
\end{lemma}
\begin{proof}
    By Lemma \ref{Lem_UncollapsedReducedToPrimes}, we only need
    to consider $H_2$ and $H_p$. By Remark 
    \ref{Rem_SelfDualUniversalCoverGroupSizes}, $|H_{2p}| \neq 2$.

    Consider the map $H_d \to \{\pm 1\}$ that maps $a$, $b$, and $c$ all
    to $-1$. It is only well--defined for even $d$, thus $H_{2p} \neq H_p$.
\end{proof}

\begin{theorem}\label{Theo_PrimePowersUncollapsedWithGAP}
    $H_{2^n}$ is uncollapsed ($n\geq 3$).
    $H_{3^n}$ is uncollapsed ($n\geq 2$).
    $H_{5^n}$ is uncollapsed ($n\geq 1$).
\end{theorem}
\begin{proof}
    This can be shown by a lengthy calculation in GAP (\cite{GAP4}),
    that is covered in \cite[Theorem 9.5.7]{Baumeister_Phd}. The code
    can also be found at 
    \url{https://markusbaumeister.github.io/code/UncollapsedGeodesicTriangleGroups.g}.
    The proof computes
    a presentation of the subgroups 
    $\llangle (bac)^8\rrangle \leq H_{2^n}$,
    $\llangle (bac)^9\rrangle \leq H_{3^n}$, 
    and $\llangle (bac)^5\rrangle \leq H_{5^n}$, by using
    the subgroup presentation algorithm in 
    \cite{HoltEickOBrien_HandbookComputationalGroupTheory}. Then, we
    calculate the abelian invariants of this subgroups to
    distinguish the groups.
\end{proof}

\subsection{Voltage assignments}\label{Subsect_VoltageAssignments}
In this subsection, we show that $H_{p^n}$ and $H_{4p}$ are uncollapsed.
To achieve this, we use
\textit{corner voltage assigments} to construct appropriate surface coverings.
The presentation of this theory follows \cite{ArchConder_TrinitySymmetry}.
For more context about voltage assignments, compare
\cite{ArchRich_BranchedCoveringsOfMaps}.

\begin{definition}
    Let $(\mathcal{F},\V{\alpha},\E{\beta},\F{\gamma})$ be a surface
    and $B$ be a group. A map $v: \mathcal{F} \to B$
    is called \textbf{corner voltage assignment}, if 
    $v(\E{\beta}.x) = v(x)^{-1}$ holds for all $x \in \mathcal{F}$. In this
    scenario, $B$ is called the \textbf{voltage group}.
\end{definition}

\begin{definition}
    Let $(\mathcal{F},\V{\alpha},\E{\beta},\F{\gamma})$ be a surface with
    corner voltage assignment $v: \mathcal{F} \to B$. The
    \textbf{lift} of $(\mathcal{F},\V{\alpha},\E{\beta},\F{\gamma})$ 
    with respect to
    $v$ is the quadruple 
    $(\mathcal{F}\times B, \V{\hat{\alpha}},\E{\hat{\beta}},\F{\hat{\gamma}})$,
    with
    \begin{align*}
        \V{\hat{\alpha}}.(x,g) &\isdef (\V{\alpha}.x,g), &
        \E{\hat{\beta}}.(x,g) &\isdef (\E{\beta}.x,v(x)g), &
        \F{\hat{\gamma}}.(x,g) &\isdef (\F{\gamma}.x,g).
    \end{align*}
\end{definition}
In general, the lift does not define a surface in the sense of
Definition \ref{Def_Surface}. There are two possible
reasons: The orbits of $\langle\V{\hat{\alpha}},\E{\hat{\beta}}\rangle$
might have an order larger than 3, and the group 
$\langle\V{\hat{\alpha}},\E{\hat{\beta}},\F{\hat{\gamma}}\rangle$
might not act transitively on $\mathcal{F}\times B$. The transitivity
can be achieved by restriction to a single orbit.

The orbit lengths of 
$\langle\V{\hat{\alpha}},\E{\hat{\beta}}\rangle$ can be controlled
by an additional condition.

\begin{remark}\label{Rem_LiftIsTriangulated}
    Let $(\mathcal{F},\V{\alpha},\E{\beta},\F{\gamma})$ be a surface with
    corner voltage assignment $v: \mathcal{F} \to B$ and lift
    $(\mathcal{F}\times B, \V{\hat{\alpha}},\E{\hat{\beta}},\F{\hat{\gamma}})$.
    Then,
    \begin{enumerate}
        \item $\V{\hat{\alpha}}$, $\E{\hat\beta}$, and $\F{\hat\gamma}$ are
            involutions without fixed points on $\mathcal{F}\times B$.
        \item $\langle\V{\hat\alpha},\E{\hat\beta},\F{\hat\gamma}\rangle$
            acts transitively on each of its orbits.
        \item $\V{\hat\alpha}\E{\hat\beta}$ consists only of 3--cycles if
            and only if $v(\E{\beta}\V{\alpha}.x)v(\V{\alpha}\E{\beta}.x)v(x) = 1$
            for all $x \in \mathcal{F}$.
        \item $\V{\hat\alpha}\F{\hat\gamma}$ consists only of 2--cycles.
    \end{enumerate}
\end{remark}
\begin{proof}
    Most properties follow from the corresponding properties for surfaces
    (compare Definition \ref{Def_Surface}). We compute
    $(\V{\hat\alpha}\E{\hat\beta})^3.(x,g)$ for $(x,g) \in \mathcal{F}\times B$:
    \begin{equation*}
        (\V{\hat\alpha}\E{\hat\beta})^2.(\V{\alpha}\E{\beta}x,v(x)g) = 
        (\V{\hat\alpha}\E{\hat\beta}).(\E{\beta}\V{\alpha}x,v(\V{\alpha}\E{\beta}.x)v(x)g) = 
        (x, v(\E{\beta}\V{\alpha}.x)v(\V{\alpha}\E{\beta}.x)v(x)g)
    \end{equation*}
    Therefore, the product condition is equivalent to 
    $(\V{\hat\alpha}\E{\hat\beta})^3=1$. If there is an element
    $(x,g)$ that does not lie in a 3--cycle of $\V{\hat\alpha}\E{\hat\beta}$,
    it has to be fixed by it. But then, $x$ would have to be
    fixed by $\V{\alpha}\E{\beta}$, contradicting that we started with a surface.
\end{proof}

Before we can construct appropriate lifts, we need to prove a few technical lemmas.
\begin{lemma}\label{Lem_DifferentFacesInUmbrellaForCoprimes}
    Let $H_d$ be infinite. Then, there is no $k$ with $\gcd(k,d) = 1$ such that
    $(bc)^k \in \langle a,b\rangle$ or $(bac)^k \in \langle a,b\rangle$.
\end{lemma}
\begin{proof}
    Without loss of generality, only consider $bc$. Since $(bc)^d = 1$, we
    can apply the \texttt{Euclidean} algorithm to deduce 
    $bc \in \langle a,b\rangle$. A short computation in GAP
    (\cite{GAP4}) shows that $H_d$ has to be finite in this case.
\end{proof}
\begin{lemma}\label{Lem_DifferentFacesInUmbrellaForPrimePower}
    Let $p\geq 5$ be a prime and $H_{p^n}$ be infinite and uncollapsed.
    If either $(bc)^k \in \langle a,b\rangle$ or $(bac)^k \in \langle a,b\rangle$
    holds, then $k$ is a multiple of $p^n$.
\end{lemma}
\begin{proof}
    By Lemma \ref{Lem_DifferentFacesInUmbrellaForCoprimes}, $k$
    cannot be coprime to $p^n$. If $k$ is not a multiple of $p^n$,
    we can reduce to the case $k = p^m$ with $0  < m < n$. In this
    case, $1 = (bc)^{p^n} = ((bc)^{p^m})^{p^{n-m}} = x^{p^{n-m}}$
    for some $x \in \langle a,b\rangle$.

    Since $p \geq 5$, the element $x$ cannot have order 2 or 3.
    The only remaining element in $\langle a,b\rangle$ is 1.
    But $x = 1$ would imply $H_{p^n} = H_{p^m}$ 
    (by Remark \ref{Rem_RelatorsInGeodesicTriangleGroup}),
    contradicting $H_{p^n}$ being uncollapsed.
\end{proof}
\begin{lemma}\label{Lem_DifferentFacesInUmbrellaForDoublePrime}
    Let $p$ be a prime and $H_{2p}$ be infinite and uncollapsed.
    If either $(bc)^k \in \langle a,b\rangle$ or $(bac)^k \in \langle a,b\rangle$
    holds, then $k$ is a multiple of $2p$.
\end{lemma}
\begin{proof}
    By Lemma \ref{Lem_DifferentFacesInUmbrellaForCoprimes}, $k$
    cannot be coprime to $2p$. If $k$ is not a multiple of $2p$,
    we can reduce to the cases $k = 2$ or $k = p$. For $k=2$,
    it is easy to check with GAP (\cite{GAP4}) that 
    every case of $(bc)^2 \in \langle a,b\rangle$ implies the
    finiteness of $H_{2p}$.

    For $k = p$, we have $1 = (bc)^{2p} = ((bc)^p)^2$. Therefore,
    $(bc)^p = 1$ or has order 2. The first case is impossible
    since $H_{2p}$ is uncollapsed. The second one implies
    $(bc)^p \in \{a,b,aba\}$. To analyse these cases, we
    use the equality $(bc)^pb(bc)^p = b(cb)^p(bc)^p = b$:
    \begin{align*}
        (bc)^p &= a &\text{implies} && 1 = (ab)^3 &= ((bc)^pb)^3 = b^2(bc)^pb = ab, \\
        (bc)^p &= aba &\text{implies} && 1 = (ab)^3 &= (bc)^pb(bc)^pa = ba.
    \end{align*}
    In this case, $H_d$ is finite.
    From $(bc)^p = b$ we can deduce $(bc)^{p-1} = c$, from which
    $(bc)^{p-2} = b$ follows. Inductively, either $b=1$ or
    $c=1$ holds, then $H_{2p}$ is finite.
\end{proof}

\begin{proposition}\label{Prop_ConstructLiftSurface}
    Let $H_d$ be uncollapsed such that $(bc)^k \in \langle a,b\rangle$
    or $(bac)^k \in \langle a,b,\rangle$ is only possible if
    $k$ is a multiple of $d$. For any prime $p$, there exists
    a degree--$dp$--surface 
    $(\mathcal{G},\V{\hat\alpha},\E{\hat\beta},\F{\hat\gamma})$
    such that $\E{\hat\beta}\V{\hat\alpha}\F{\gamma}$ consists
    only of $dp$-cycles.
    In particular, $H_d \neq H_{dp}$.
\end{proposition}
\begin{proof}
    To construct the surface, we start with the geodesic self--dual
    degree--$d$--surface $(\rcosets{H_d}{\{1\}},a,b,c)$. Let
    $F$ be the set of $\langle a,b\rangle$--orbits (the faces).
    Choose one element $f$ from each orbit to represent the orbit
    as $\langle a,b\rangle.f$. Define the voltage group
    \begin{equation*}
        B \isdef \begin{cases}
            (\Z/p\Z)^F & p \neq 2 \\
            (V_4)^F & p = 2 \text{, with } V_4 = \langle s,t\mid s^2,t^2,(st)^2\rangle
        \end{cases}
    \end{equation*}
    and the corner voltage assignment $v$
    as follows: $v(x)$ only has a non--trivial value in the
    component $\langle a,b\rangle.f_x$, where 
    $x \in \langle a,b\rangle.f_x$. For the elements of this
    orbit, its value is defined as
    \begin{align*}
        f_x &\mapsto 1   & ab.f_x &\mapsto 1 & abab.f_x &\mapsto p-2 \\
        b.f_x &\mapsto p-1   & bab.f_x &\mapsto p-1 & babab.f_x &\mapsto 2
    \end{align*}
    for odd $p$ and as
    \begin{align*}
        f_x &\mapsto s   & ab.f_x &\mapsto t & abab.f_x &\mapsto st \\
        b.f_x &\mapsto s   & bab.f_x &\mapsto t & babab.f_x &\mapsto st
    \end{align*}
    for $p=2$.
    By Remark \ref{Rem_LiftIsTriangulated}, the lift via $v$
    produces a surface 
    $(\mathcal{G},\V{\hat\alpha},\E{\hat\beta},\F{\hat\gamma})$
    (after restriction to one orbit of 
    $\langle \V{\hat\alpha},\E{\hat\beta},\F{\hat\gamma}\rangle$ 
    on $\rcosets{H_d}{\{1\}} \times B$).

    We compute the cycle lengths of $\E{\hat\beta}\F{\hat\gamma}$ 
    (umbrellas)
    and $\E{\hat\beta}\V{\hat\alpha}\F{\hat\gamma}$ (geodesics). Since
    the argument for them is similar, we only give the case for
    $\E{\hat\beta}\F{\hat\gamma}$.

    Let $\hat{f} = (x,g) \in \mathcal{G} \subset \rcosets{H_d}{\{1\}} \times B$.
    If $\langle a,b\rangle.(bc)^kx = \langle a,b\rangle.x$, we
    conclude $(bc)^k \in \langle a,b\rangle$. By assumption, this
    is only possible if $k$ is a multiple of $d$. In particular,
    $(\E{\hat\beta}\F{\hat\gamma})^d.\hat{f} = (x,wg)$, with $w \in B$
    such that $w$ has a non--trivial entry at each 
    $(\E{\hat\beta}\F{\hat\gamma})^k.x$
    (for $0 \leq k < d$). By the previous analysis, these positions are
    all distinct. Since all non--trivial elements in $\Z/p\Z$ have order
    $p$ and all non--trivial elements in $V_4$ have order 2, 
    the order of $\E{\hat\beta}\F{\hat\gamma}$ is $dp$.

    To show the additional claim, observe that $T_{dp}$ acts transitively
    on $\mathcal{G}$ (Remark \ref{Rem_TriangleGroupActsOnFlags}). The
    element $(bac)^{dp} \in T_{dp}$ acts trivially, thus
    $H_{dp} = T_{dp}/\llangle (bac)^{dp}\rrangle$ also acts transitively
    on $\mathcal{G}$. But the $bac$--orbits of $H_d$ have maximal
    length $d$, so $H_d \neq H_{dp}$.
\end{proof}

\begin{proposition}\label{Prop_HpnUncollapsed}
    $H_{p^n}$ is uncollapsed for $p > 3$ prime and $n > 1$.
\end{proposition}
\begin{proof}
    We show the claim by induction. By Corollary 
    \ref{Cor_UncollapsedByFiniteFactor}, $H_{25}$ and $H_{49}$
    are uncollapsed, together with all $H_p$ with $p > 10$ prime.
    Also, all of them are infinite by Remark 
    \ref{Rem_SelfDualUniversalCoverGroupSizes} and satisfy the 
    assumption of proposition \ref{Prop_ConstructLiftSurface}
    by Lemma \ref{Lem_DifferentFacesInUmbrellaForPrimePower}.
\end{proof}

\begin{proposition}\label{Prop_HFourPUncollapsed}
    $H_{4p}$ is uncollapsed for all odd primes $p > 3$.
\end{proposition}
\begin{proof}
    By Lemma \ref{Lem_UncollapsedReducedToPrimes}, we only have
    to consider $H_4$ and $H_{2p}$. From Remark
    \ref{Rem_SelfDualUniversalCoverGroupSizes}, clearly $H_4 \neq H_{4p}$.
    Further, $H_{2p}$ is infinite ($p > 3$), uncollapsed (Lemma 
    \ref{Lem_DoublePrimeIsUncollapsed}) and fulfills the assumption
    of Proposition \ref{Prop_ConstructLiftSurface} (Lemma 
    \ref{Lem_DifferentFacesInUmbrellaForDoublePrime}).
\end{proof}

\subsection{Uncollapsed induction}\label{Subsection_UncollapsedInduction}
After several partial results, we show that the remaining cases follow
inductively. The central observation is the following lemma:
\begin{lemma}\label{Lem_ReduceByCommonFactor}
    Let $d = k \cdot p$ for a prime $p$ such that there is a $z \mid k$
    with $p \nmid z$. Then,
    $H_d = H_k$ implies $H_{\frac{d}{z}} = H_{\frac{k}{z}}$.
\end{lemma}
\begin{proof}
    Clearly, $\frac{k}{z} \mid \gcd(k, \frac{kp}{z})$. 
    If $p^n \mid k$, then $p^{n+1} \mid \frac{kp}{z}$, since
    $p \nmid z$. Therefore, $\frac{k}{z} = \gcd(k,\frac{kp}{z})$ and
    we have
    $H_{\frac{d}{z}} = H_d / \llangle (bc)^{\frac{d}{z}}, (bac)^{\frac{d}{z}}\rrangle = H_k / \llangle (bc)^{\frac{d}{z}}, (bac)^{\frac{d}{z}}\rrangle = H_{\frac{k}{z}}$.
\end{proof}

\begin{theorem}\label{Theo_HdUncollapsed}
    $H_d$ is uncollapsed for all $d \geq 5$ with $d \neq 7$.
\end{theorem}
\begin{proof}
    Assume $H_d$ is a counterexample.
    By Lemma \ref{Lem_UncollapsedReducedToPrimes}, there is a prime $p$
    with $d = p^{n+1}z$ with $n \geq 0$ and $p \nmid z$,
    such that $H_{p^nz} = H_d$.

    By Lemma \ref{Lem_ReduceByCommonFactor}, this implies 
    $H_{p^n} = H_{p^{n+1}}$. We distinguish several cases:
    \begin{itemize}
        \item $p=2$: By Theorem \ref{Theo_PrimePowersUncollapsedWithGAP},
            this is only possible if $n \in \{0,1\}$. Corollary
            \ref{Cor_UncollapsedByFiniteFactor} restricts this further
            to $n = 1$.
        \item $p=3$: By Theorem \ref{Theo_PrimePowersUncollapsedWithGAP},
            this is only possible for $n = 0$.
        \item $p=7$: By Proposition \ref{Prop_HpnUncollapsed}, this
            is only possible for $n=0$.
        \item In all other cases, the combination of Theorem
            \ref{Theo_PrimePowersUncollapsedWithGAP}, Proposition \ref{Prop_HpnUncollapsed}
            and Corollary \ref{Cor_UncollapsedByFiniteFactor} makes
            this situation impossible.
    \end{itemize}
    For the three remaining cases, we apply Lemma \ref{Lem_ReduceByCommonFactor}
    to $d = p^{n+1}z$ in a different way: Let $q$ be a prime dividing $z$.
    Then we can use Lemma \ref{Lem_ReduceByCommonFactor} to divide
    by $\frac{z}{q}$. This gives the three cases
    $H_{4q} = H_{2q}$, $H_{3q} = H_q$, and $H_{7q} = H_q$. 
    
    The first one
    is impossible by Proposition \ref{Prop_HFourPUncollapsed} (for $p>3$)
    and Corollary \ref{Cor_UncollapsedByFiniteFactor} (for $p = 3$).
    If $q \in \{3,5,7\}$, the impossibility of the 
    other cases follow from Corollary
    \ref{Cor_UncollapsedByFiniteFactor}. Otherwise, $H_q$ satisfies
    the assumptions of Proposition \ref{Prop_ConstructLiftSurface}
    (infinite by Remark \ref{Rem_SelfDualUniversalCoverGroupSizes},
    so Lemma \ref{Lem_DifferentFacesInUmbrellaForPrimePower} holds).
\end{proof}

\section{Classification}\label{Sect_Classification}
In this section, we complete the proof of the main theorem.
Afterwards, we give a complete classification of all
geodesic self--dual degree--$d$--surfaces for $d < 10$. Several of
these surfaces are also available in the GAP--package 
\texttt{SimplicialSurfaces} (\cite{SimplicialSurfaces}). In these
cases, we will also give the command to generate this particular surface.

Recall Definition \ref{Def_GeodesicTriangleGroup} of the geodesic
triangle group $H_d$ and Definition \ref{Def_GeodesicAutomorphism}
of the geodesic automorphism $\#: H_d \to H_d$.
\begin{theorem}\label{Theo_GeodesicSelfDual}
    Let $V \leq H_d$ with $d \geq 5$ and $d \neq 7$.
    Then
    $(\rcosets{H_d}{V},a,b,c)$ is a geodesic self--dual degree--$d$--surface
    if and only if
    \begin{itemize}
        \item $g^{-1}Vg \cap X = \{1\}$ for $X \in \{\langle a\rangle,\langle c\rangle,\langle ab\rangle,\langle ac\rangle,\langle bc\rangle\}$
            and all $g \in H_d$.
        \item $\geoAut{V}$ is conjugate to $V$.
    \end{itemize}
    Furthermore, all geodesic self--dual degree--$d$--surfaces have this form.
\end{theorem}
\begin{proof}
    This follows from Corollary \ref{Cor_GeodesicSurfaceSubgroupCharacterisation}
    and Theorem \ref{Theo_HdUncollapsed}, by using the reformulation
    from Lemma \ref{Lem_UncollapsedConditionForSelfDuality}
    and Definition \ref{Def_Uncollapsed}.
\end{proof}

For $d \in \{5,6,8,9\}$, the group $H_d$ is finite, so 
we can use GAP (\cite{GAP4})
to compute all geodesic self--dual degree--$d$--surfaces.

\begin{example}
    For $d = 5$, there is only one geodesic self--dual surface, since
    only the trivial subgroup $\{1\}$ satisfies Theorem 
    \ref{Theo_GeodesicSelfDual}.
    This defines the projective plane on
    10 triangles (6 vertices and 15 edges), shown in figure  
    \ref{Fig_SelfDual5}.
    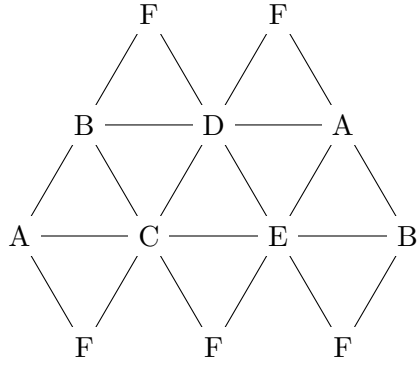
\begin{figure}[htb]
        \begin{center}
            \begin{tikzpicture}[scale=1.7]
                \coordinate (Z) at (0,0);
                \coordinate (P0) at (0:1);
                \coordinate (P1) at (60:1);
                \coordinate (P2) at (120:1);
                \coordinate (P3) at (180:1);
                \coordinate (P4) at (240:1);
                \coordinate (P5) at (300:1);
                \coordinate (Q34) at ($(P3)+(P4)$);
                \coordinate (Q44) at ($2*(P4)$);
                \coordinate (Q45) at ($(P4)+(P5)$);
                \coordinate (Q55) at ($2*(P5)$);
                \coordinate (Q05) at ($(P0)+(P5)$);

                \draw (Q34) -- (Q44) -- (P1) -- (Q05) -- (Q34) -- (P2) -- (Z) -- (P3) -- (Q45) -- (P0) -- (Z) -- (Q55) -- (Q05);
                \foreach \p/\n in {Z/D, P0/A, P1/F, P2/F, P3/B, P4/C, P5/E, Q34/A, Q44/F, Q45/F, Q55/F, Q05/B}
                    \node[fill=white] at (\p) {\n};
            \end{tikzpicture}
        \end{center}
        \caption{Geodesic self--dual degree--5--surface}
        \label{Fig_SelfDual5}
    \end{figure}
    The associated command in the \texttt{SimplicialSurfaces}--package
    is \texttt{AllGeodesicSelfDualSurfaces(10)[1]}.
\end{example}

\begin{example}
    For $d = 6$, there are exactly two geodesic self--dual surfaces, since
    there are exactly two surface subgroups satisfying Theorem 
    \ref{Theo_GeodesicSelfDual} (both are tori):
    \begin{enumerate}
        \item The trivial subgroup $\{1\}$, defining a
            surface with 18 faces (9 vertices and 27 edges).
            Its command is \texttt{AllGeodesicSelfDualSurfaces(18)[1]}.
        \item A normal subgroup of size 3, defining a surface
            with 6 faces (3 vertices and 9 edges).
            Its command is \texttt{AllGeodesicSelfDualSurfaces(6)[1]}.
    \end{enumerate}
    \begin{figure}[htb]
        \begin{center}
            \begin{tikzpicture}[scale=1.5]
                \begin{scope}[shift={(-6,0)}]
                    \coordinate (x) at (1,0);
                    \coordinate (y) at (60:1);
                    \foreach \i in {0,1,2,3}{
                        \foreach \j in {0,1,2,3}{
                            \coordinate (P\i\j) at ($\i*(x)+\j*(y)$);
                        }
                    }
                    \foreach \i in {0,1,2,3}{
                        \foreach \u/\v in {0/1,1/2,2/3}{
                            \draw (P\i\u) -- (P\i\v);
                            \draw (P\u\i) -- (P\v\i);
                        }
                    }
                    \draw (P01) -- (P10) (P02) -- (P20) (P03) -- (P30) (P13) -- (P31) (P23) -- (P32);
                    \foreach \u/\v/\l in {0/0/A, 1/0/B, 2/0/C, 3/0/A, 0/1/D, 3/1/D, 0/2/E, 3/2/E, 0/3/A, 1/3/B, 2/3/C, 3/3/A}
                        \node[fill=white] at (P\u\v) {\l};
                \end{scope}

                \begin{scope}[shift={(0.5,0.9)}]
                    \coordinate (Z) at (0,0);
                    \coordinate (P5) at (-60:1);
                    \coordinate (P0) at (0:1);
                    \coordinate (P1)  at (60:1);
                    \coordinate (P2) at (120:1);
                    \coordinate (P3) at (180:1);
                    \coordinate (Q01) at ($(P0)+(P1)$);
                    \coordinate (Q12) at ($(P2)+(P1)$);

                    \draw (P5) -- (Q01) -- (P2) -- cycle;
                    \draw (P0) -- (Q12) -- (P3) -- cycle;
                    \draw (Z) -- (P1);

                    \foreach \p/\l in {Z/B, P5/A, P0/C, P1/A, P2/C, P3/A, Q01/B, Q12/B}
                        \node[fill=white] at (\p) {\l};
                \end{scope}
            \end{tikzpicture}
        \end{center}
        \caption{Geodesic self--dual degree--6--surfaces}
    \end{figure}
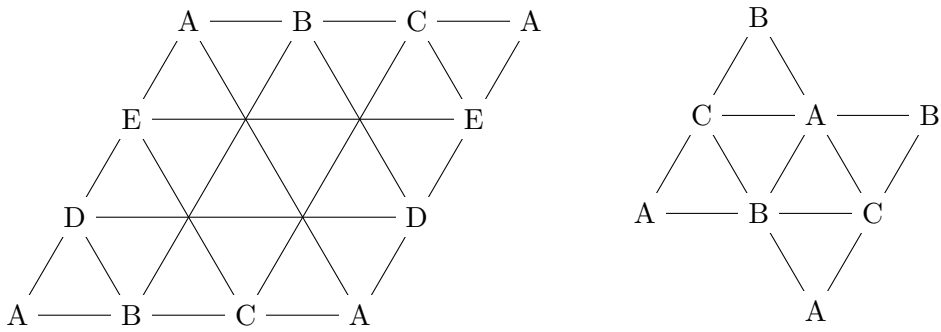
\end{example}

\begin{example}
    For $d = 8$, there are exactly four geodesic self--dual surfaces, defined
    by the four surface subgroups satisfying Theorem \ref{Theo_GeodesicSelfDual}
    (up to conjugation):
    \begin{enumerate}
        \item The trivial subgroup $\{1\}$, defining an orientable 
            surface with 42 vertices, 168 edges, and 112 faces
            (genus 8).

            Its command is \texttt{AllGeodesicSelfDualSurfaces(112)[1]}.
        \item A normal subgroup of size 2, defining a non--orientable
            surface with
            21 vertices, 84 edges, and 56 faces (genus 8).

            Its command is \texttt{AllGeodesicSelfDualSurfaces(56)[1]}.
        \item A subgroup of size 7 and index 96, defining an orientable
            surface with 6 vertices, 24 edges, and 16 faces
            (genus 2).
        \item A subgroup of size 14 and index 48, defining a non--orientable
            surface with 3 vertices, 12 edges, and 8 faces 
            (genus 2).
    \end{enumerate}
\end{example}

\begin{example}
    For $d = 9$, there are exactly three geodesic self--dual surfaces,
    defined by the three surface subgroups satisfying Theorem
    \ref{Theo_GeodesicSelfDual} (up to conjugation):
    \begin{enumerate}
        \item The trivial subgroup $\{1\}$, defining a non--orientable 
            surface with 190 vertices, 855 edges, and 570 faces (genus 96).

            Its command is \texttt{AllGeodesicSelfDualSurfaces(570)[1]}.
        \item A group of size 5 with index 684, defining a non--orientable
            surface with 38 vertices, 171 edges, and 114 faces (genus 20).
        \item A group of size 19 with index 180, defining a non--orientable
            surface with 10 vertices, 45 edges, and 30 faces (genus 6).
    \end{enumerate}
\end{example}

\section{Summary}
In this paper, we have characterised all geodesic self--dual
degree--$d$--surfaces. For $d < 10$, we classified all of them.
For $d \geq 10$, the situation is unclear: We conjecture that there are
infinitely many geodesic self--dual surfaces for each $d \geq 10$. This is 
based on the observation
that our calculations reached their computational limits before stopping
to construct further examples.

Unfortunately, it is still unclear whether the geodesic surface
subgroups of $H_d$ for $d \geq 10$ can be characterised in a
fashion that is more amenable to analysis.

\section*{Acknowledgements}
I am grateful for the support of my supervisor, Alice Niemeyer, in 
writing this paper. I am also grateful for the funding by 
the ``Graduiertenkolleg Experimentelle konstruktive Algebra'' (number 1632)
during the writing of this paper.

\cleardoublepage
\addcontentsline{toc}{section}{References}
\providecommand{\bysame}{\leavevmode\hbox to3em{\hrulefill}\thinspace}
\providecommand{\MR}{\relax\ifhmode\unskip\space\fi MR }
\providecommand{\MRhref}[2]{%
  \href{http://www.ams.org/mathscinet-getitem?mr=#1}{#2}
}
\providecommand{\href}[2]{#2}


\end{document}